\documentclass[12pt]{article}
\usepackage{amsthm,amsmath,amssymb,color}
\usepackage{array}
\usepackage{ytableau}
\usepackage{kotex}
\usepackage[shortlabels]{enumitem}
\usepackage{mathtools}
\usepackage{subfig}
\usepackage{graphicx}
\usepackage{caption,color}  
\usepackage{epic}
\usepackage{setspace}
\usepackage{float}
\usepackage{tikz}
\usepackage{multirow}
\usepackage{hyperref}
\usepackage{graphics}
\usepackage{mathtools}
\usepackage{nicematrix}
\usepackage{comment}
\usepackage[OT2,OT1]{fontenc}
\usetikzlibrary{graphs,graphs.standard}
\tikzgraphsset{edges={draw,semithick}, nodes={circle,draw,semithick}}
\usepackage{float}
\usepackage[normalem]{ulem}

\newcommand{\ret}{\textcolor{red}}

\newtheorem{thm}{Theorem}[section]
\newtheorem{theorem}[thm]{Theorem}
\newtheorem{lem}[thm]{Lemma}
\newtheorem{prop}[thm]{Proposition}
\newtheorem{problem}{Problem}

\newtheorem{defi}[thm]{Definition}

\theoremstyle{remark}
\newtheorem{case}{Case}
\newtheorem{subcase}{Subcase}[case]

\begin{document}
\begin{spacing}{1.15}
\title{\textbf{{\Large ~Rainbow Hamiltonicity and the spectral radius}} \author{Yuke Zhang\thanks{School of Mathematical Sciences, Zhejiang University of Technology, Hangzhou, 310023, P.R.~China. Email:~zhang\_yk1029@163.com}\and Edwin R. van Dam\thanks{Department of Econometrics and O.R., Tilburg University, Tilburg, 5000 LE, the Netherlands. Email:~Edwin.vanDam@tilburguniversity.edu}}}
\date{}
\maketitle

\begin{abstract} 
Let $\mathcal{G}=\{G_1,\ldots,G_n \}$ be a family of graphs of order $n$ with the same vertex set. A rainbow Hamiltonian cycle in $\mathcal{G}$ is a cycle that visits each vertex precisely once such that any two edges belong to different graphs of $\mathcal{G}$. { We show that if each $G_i$ has more than $\binom{n-1}{2}+1$ edges, then $\mathcal{G}$ admits a rainbow Hamiltonian cycle and pose the problem of characterizing rainbow Hamiltonicity under the condition that all $G_i$ have at least $\binom{n-1}{2}+1$ edges.} 
Towards a solution of that problem, we give a sufficient condition for the existence of a rainbow Hamiltonian cycle in terms of the spectral radii of the graphs in $\mathcal{G}$ and completely characterize the corresponding
extremal graphs.
\\

	

\end{abstract}

\section{Introduction}

A graph $G$ is called \textit{Hamiltonian} if it contains a cycle that visits each vertex of $G$ exactly once. Determining Hamiltonicity of graphs is an old and classic problem in graph theory, which is known to be NP-hard \cite[pp. 85--103]{karp2010reducibility}. 
In 1952, Dirac \cite{dirac1952some}  gave a sufficient condition in terms of the minimum degree, i.e., every graph $G$ with minimum degree at least $\tfrac{|V(G)|}{2}$ contains a Hamiltonian cycle. In 1960, Ore \cite{ore1960note} extended Dirac's condition and proved that every graph $G$ with $\sigma_2(G)\ge |V(G)|$ is Hamiltonian, where $\sigma_2(G)$ is the minimum degree sum among all pairs of nonadjacent vertices of $G$. These two conditions are respectively called \textit{Dirac-type condition} and \textit{Ore-type condition}. 
In 1961, Ore \cite{ore1961arc} derived another well-known condition in terms of the size (the number of edges $e(G)$ of $G$).
\begin{theorem}\cite[Theorem 4.3]{ore1961arc}\label{oresize}
	Let $G$ be a graph of order $n$. If $e(G)> \binom{n-1}{2}+1$, then $G$ is Hamiltonian.
\end{theorem}
{We note that Theorem \ref{oresize} follows from the Ore-type condition \cite{ore1961arc}.}
Denote by $\vee$ and $\cup$ the \textit{join} and \textit{union} of two graphs, respectively. 
Ore \cite{ore1961arc} noted that his result is best possible in that $K_1\vee (K_{n-2}\cup K_1)$ and $K_2\vee 3K_1$ are non-Hamiltonian and have exactly $\binom{n-1}{2}+1$ edges. Bondy \cite{bondy1971} showed that these are the only extremal exceptional graphs.

\begin{theorem}\cite[Theorem 1]{bondy1971}\label{bondysize}
	Let $G$ be a graph of order $n$. If $e(G)\ge \binom{n-1}{2}+1$, then $G$ is Hamiltonian, unless $G \cong K_1\vee (K_{n-2}\cup K_1)$ or $K_2\vee 3K_1$.
\end{theorem}
The \textit{adjacency matrix} of $G$ with vertex set $V(G)$ and edge set $E(G)$ is defined as the matrix $A(G) =(a_{u,v})_{u,v\in  V(G)}$ with $a_{u,v}=1$ if $uv\in E(G) $, and $a_{u,v}=0$ otherwise. The largest eigenvalue of $A(G)$, denoted by $\rho (G)$, is called the \textit{spectral
	radius} of $G$.
In 2010, Fiedler and Nikiforov \cite{fiedler2010spectral} gave a sufficient condition for Hamiltonicity from a spectral perspective:
\begin{theorem}\cite[Theorem 2]{fiedler2010spectral}\label{fnspectral}
	Let $G$ be a graph of order $n$. If
	$\rho(G)>n-2,$
	then $G$ is Hamiltonian unless $G\cong K_1\vee (K_{n-2}\cup K_1).$
\end{theorem}
We note though that this result is weaker than Theorem \ref{bondysize},
for it follows using Stanley's inequality $e(G) \geq \rho(G)(\rho(G)+1)/2$ \cite{Stanley}.
Let $\mathcal{G}=\{G_1,\ldots,G_n \}$ be a family of graphs with the same vertex  set $V$. Informally, we think of $G_i$ as a graph with edges that have color $i$.
Now a graph $H$ is called \textit{rainbow} in $\mathcal{G}$ if any two edges of $H$ belong to different graphs of $\mathcal{G}$; informally, any two edges have a different color. 

For many famous results, rainbow versions have been considered, which led to beautiful results, such as rainbow versions for Carath{\'e}odory's theorem \cite{barany1982generalization}, the Erd{\H{o}}s--Ko--Rado theorem \cite{aharoni2017rainbow}, Mantel's theorem \cite{aharoni2020rainbow}, and more  \cite{bollobas2004multicoloured,das2013rainbow,holmsen2008points,kalai2005topological}. Also rainbow matchings have been considered. In particular, 
Guo et al.~\cite{guo2022spectral} obtained conditions in terms of the spectral radii for the existence of a rainbow matching.

In this paper, we let $H$ be a Hamiltonian cycle, and say that $\mathcal{G}$ is \textit{rainbow Hamiltonian} if it admits such a rainbow Hamiltonian cycle. Here we assume that there are $n$ graphs in $\mathcal{G}$, each on the same vertex set of size $n$.
Chen, Wang, and Zhao \cite{cheng2021rainbow} obtained a rainbow version of Dirac's theorem and showed that if each of the graphs $G_i\in \mathcal{G}$ has smallest degree are least $(1/2+o(1))n$, then $\mathcal{G}$ admits a rainbow Hamiltonian cycle. Joos and Kim \cite{joos2020rainbow} strengthened this result to smallest degree at least $n/2$. Li, Li, and Li \cite{li2023rainbow} investigated the existence of a rainbow Hamiltonian path under an Ore-type condition: if  $\sigma_2(G_i)>  n-2$ for all $i$, then $\mathcal{G}$ admits a rainbow Hamiltonian path \cite[Theorem 1.3]{li2023rainbow}. {A spectral condition for a
rainbow Hamiltonian path was obtained by He, Li and Feng \cite[Theorem 2]{he2024spectral}.}
For other interesting results on rainbow Hamiltonian cycles and related problems we refer the reader to \cite{aharoni2019rainbow,lu2021rainbow,montgomery2021transversal,tang2023rainbow}. 

In this paper, we consider rainbow versions of Theorems \ref{oresize},\ref{bondysize}, and \ref{fnspectral}. Our first main result is the following rainbow version of Ore's size result, thus obtaining a sufficient condition on the sizes of the graphs for the existence of a rainbow Hamiltonian cycle. 
{\begin{theorem}\label{thmsize}
	Let $n\ge 4$ and $\mathcal{G}=\{G_1,G_2,\ldots,G_n\}$ be a family of graphs on the same vertex set $[n]$, with
	\begin{align*}
		e(G_i)>\binom{n-1}{2}+1, \mbox{\ for\ $i=1,2,\ldots,n.$}
	\end{align*}
	Then $\mathcal{G}$ admits a rainbow Hamiltonian cycle.
\end{theorem}}

{Next, we would like to strengthen this result, and obtain a rainbow version of Theorem \ref{bondysize}, by allowing $e(G_i) \ge \binom{n-1}{2}+1$. For each $n \ge 6$, there is one exceptional graph, $K_1\vee (K_{n-2}\cup K_1)$, that is not Hamiltonian, and hence the graph family $\mathcal{G}=\{G_1,G_2,\ldots,G_n\}$ with $G_1=\cdots=G_n\cong K_1\vee (K_{n-2}\cup K_1)$ (see Fig.~\ref{extgraphs}) does not admit a rainbow Hamiltonian cycle.}

\begin{figure}[htp]
	\centering
	\includegraphics[scale=0.45] {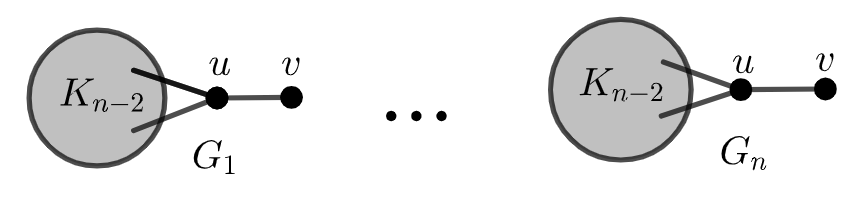}
	\caption{Graph family with  $G_1 = G_2=\cdots= G_n \cong  K_1\vee (K_{n-2}\cup K_1)$ that is not rainbow Hamiltonian.}
	\label{extgraphs}
\end{figure}

{Fortunately, any other graph family with graphs isomorphic to 
$K_1\vee (K_{n-2}\cup K_1)$ is rainbow Hamiltonian, as we shall show in Proposition \ref{extgraphlem}. Thus, we pose the following problem based on Theorems \ref{bondysize} and \ref{thmsize}:
\begin{problem}\label{problem}
	Let {$n\ge 6$} and $\mathcal{G}=\{G_1,G_2,\ldots,G_n\}$ be a family of graphs on the same vertex set $[n]$, with
	\begin{align*}
		e(G_i)\ge \binom{n-1}{2}+1, \mbox{\ for\ $i=1,2,\ldots,n.$}
	\end{align*}
	Does $\mathcal{G}$ admit a rainbow Hamiltonian cycle unless $G_1 = G_2=\cdots= G_n \cong K_1\vee (K_{n-2}\cup K_1)$ ?
\end{problem}}
{We remark that for $n=5$, the family $G_1 = G_2=\cdots= G_5 \cong K_2\vee 3K_1$ is clearly also not rainbow Hamiltonian.}  

{A possible step towards solving Problem \ref{problem} is to replace the size condition by a spectral condition and to obtain a rainbow version of the result by Fiedler and Nikiforov (Theorem \ref{fnspectral}). Indeed, recall that if $\rho(G)>n-2$, then $e(G)\ge \binom{n-1}{2}+1$ by Stanley's inequality.
We note that the spectral radius of $K_1\vee (K_{n-2}\cup K_1)$ is (just slightly) larger than $n-2$, because it has $K_{n-1}$ as a proper subgraph (and it is connected) \cite[Prop.~1.3.10]{cvetkovic2009introduction}. Thus, our second main result is the following.
\begin{theorem}\label{thmrho}
	Let $n\ge 4$ and $\mathcal{G}=\{G_1,G_2,\ldots,G_n\}$ be a family of graphs on the same vertex set $[n]$, with
	\begin{align*}
		\rho(G_i)> n-2,  \mbox{\ for\ $i=1,2,\ldots,n.$}
	\end{align*}
	Then $\mathcal{G}$ admits a rainbow Hamiltonian cycle unless $G_1 = G_2=\cdots= G_n \cong  K_1\vee (K_{n-2}\cup K_1).$
\end{theorem}}

Our paper is further organized as follows: In Section \ref{sec:preliminaries}, we describe the so-called
Kelmans transformation, which is a well-known technique in extremal graph theory, and show that if a graph family is rainbow Hamiltonian after a Kelmans transformation, then it must be rainbow Hamiltonian itself (Lemma \ref{shift_ham}).
In Section \ref{Sec:size}, we focus on conditions on the size of graphs and prove Theorem \ref{thmsize}. 
In Section \ref{sec:spectralradius}, we obtain conditions on the spectral radii of graphs for the existence of a rainbow Hamiltonian cycle, and prove Theorem \ref{thmrho}.
In the final section, we obtain a similar result in terms of the signless Laplacian spectral radii (Theorem \ref{thmsignless}).

\section{The Kelmans transformation}\label{sec:preliminaries}

In extremal graph theory, one of the most important and widely used tools is the Kelmans transformation, as it controls efficiently many graph parameters. In this section
we will state and prove some results about this transformation that are relevant to obtain our main results. For more about the Kelmans transformation, we refer the reader to \cite{csikvari2011applications,kelmans1981}.
Let $x , y\in V(G)$, and let $N_G(x)$ and $N_G(y)$ denote the neighborhoods in $G$ of vertices $x$ and $y$, respectively.

\begin{defi}\cite[Kelmans transformation]{kelmans1981}
	The Kelmans transformation on $G$ from $x$ to $y$ produces a new graph $\mathcal{K}_{x y}(G)$ by erasing all edges between $y$ and
	$N_G(y)\backslash (N_G(x)\cup \{x\})$ and adding all edges between $x$ and $N_G(y)\backslash (N_G(x)\cup \{x\})$.
\end{defi}

Thus, the Kelmans transformation does not change the number of edges. However, it does affect 
the spectral radius $\rho (G)$.

\begin{lem}\cite[Theorem 2.1]{csikvari2009conjecture} \label{ktrho}
	Let $G$ be a graph and $x, y\in V(G)$. Then $\rho(\mathcal{K}_{xy}(G)) \ge \rho(G)$.
\end{lem}

Note that $\mathcal{K}_{xy}(G)\cong \mathcal{K}_{yx}(G)$, where we denote $G_1 \cong G_2$ if $G_1$ and $G_2$ are isomorphic. We write $G_1 = G_2$ if $V (G_1) = V (G_2)$ and $E(G_1) = E(G_2)$.

For a positive integer $n$, we use $[n]$ to denote the set $\{1, 2, \ldots ,n \}$. Let $G$ be a graph with vertex set $[n]$.
Iterating the Kelmans transformation on $G$ for all vertex pairs $ (x, y)$ satisfying $x < y$ will eventually produce a so-called threshold graph, which is invariant under all such Kelmans transformations. {It is well known that the obtained threshold graph, which we denote by $\mathcal{K}(G)$, is independent of the order in which the transformations are performed (and in fact, Frankl \cite{franklshifting} showed that at most $\binom{n}{2}$ iterations are necessary).} It follows that $N_{\mathcal{K}(G)}(y)\backslash \{x\} \subseteq  N_{\mathcal{K}(G)}(x)\backslash \{y\}$ for each $x<y$. We will use this property mostly as follows.

\begin{lem}\label{lem:kelordering}
	Let $x<y$ and $z \neq x,y$. If $xz \notin E(\mathcal{K}(G))$, then $yz \notin E(\mathcal{K}(G))$.
\end{lem}

Given a family $\mathcal{G}$, we let
$\mathcal{K}_{x y}(\mathcal{G})=\{ \mathcal{K}_{x y}({G}): G\in \mathcal{G}\}$ and
$\mathcal{K}(\mathcal{G})=\{ \mathcal{K}({G}): G\in \mathcal{G}\}.$
The next result shows that if the Kelmans transformed graph family $\mathcal{K}(\mathcal{G})$ is rainbow Hamiltonian then so is $\mathcal{G}$.
\begin{lem}\label{shift_ham}
	Let $\mathcal{G}=\{G_1,G_2,\ldots,G_n \}$ be a family of graphs on vertex set $[n]$. If $\mathcal{K}(\mathcal{G})$ admits a rainbow Hamiltonian cycle, then so does $\mathcal{G}$.
\end{lem}

\begin{proof}
	It suffices to prove that for arbitrary vertices $x,y\in  V(G)$, if $\mathcal{K}_{xy}(\mathcal{G})$ admits a rainbow Hamiltonian cycle, then so does $\mathcal{G}$. Let $C^R$ be a rainbow Hamiltonian cycle of $\mathcal{K}_{xy}(\mathcal{G})$ where $E(C^R)=\{e_i: e_i\in \mathcal{K}_{xy}(G_i),  1\le i \le n \}$.

	\setcounter{subcase}{0}
	\begin{case}
		$xy\in  E(C^R)$.
	\end{case}
	Suppose  that  $e_i=ux$ and $e_j=yv$ are two other edges in $C^R$. Because $v$ is a neighor of $y$ in ${K}_{xy}(G_j)$, $v$ must be a common neighbor of $x$ and $y$ in both $\mathcal{K}_{xy}(G_j)$ and $G_j$. Thus,  $e_j=yv\in E(G_j)$ and $xv\in E(G_j)$. 
	If $uy\in E(G_i)$, then we find a rainbow Hamiltonian cycle of $\mathcal{G}$ by replacing $e_i$ and $e_j$ by $uy$ and $xv$ (see Fig. \ref{lmm2.5switch}). If  $uy\notin E(G_i)$, then $ux\in E(G_i)$ (otherwise, $u\notin N_{G_i}(x)$ and $u\notin N_{G_i}(y)$ and then $e_i=ux\notin E(\mathcal{K}_{xy}(G_i)) $, a contradiction), so  $C^R$ is also a rainbow Hamiltonian cycle of $\mathcal{G}$.
	
	\begin{case}
		$xy\notin  E(C^R)$.
	\end{case}
	Let $e_i=a_1x$, $e_{i+1}=xa_2$, $e_j=b_1y$, and $e_{j+1}=yb_2$  be edges in $C^R$ (where $a_1,a_2,b_1,$ and $b_2 $ are not all necessarily distinct). Similar as in the above case, it follows that 
	$b_1$ is a common neighbor of $x$ and $y$ in both $\mathcal{K}_{xy}(G_j)$ and $G_j$, so that $b_1x, b_1y \in  E(G_j)$. Similarly, we have $xb_2, yb_2 \in E(G_{j+1})$. 
	\begin{subcase} \label{exist1}
		$ya_1\notin E(G_i) $.
	\end{subcase}
	If $ya_1\notin E(G_i) $, then $a_1x\in E(G_i) $ (otherwise $e_i=a_1x\notin E(\mathcal{K}_{xy}(G_i)) $, a contradiction). (i) If $xa_2 \in E(G_{i+1})$, then $C^R$ is a rainbow Hamiltonian cycle of $\mathcal{G}$; (ii) If $xa_2 \notin E(G_{i+1})$, then $ya_2\in E(G_{i+1})$ (otherwise $e_{i+1}=xa_2\notin E(\mathcal{K}_{xy}(G_{i+1})) $, a contradiction), and we obtain a rainbow Hamiltonian cycle by replacing $e_{i+1}$ by $ya_2$ and $e_{j}$ by $b_1x$ (see Fig. \ref{lmm2.5switch}).
	
	\begin{subcase} \label{exist2}
		$ya_1\in E(G_i) $.
	\end{subcase}
	If $ya_1\in E(G_i) $, then we first replace $e_{i}$ by $ya_1$ and $e_{j+1}$ by $xb_2$.  (i) If $xa_2 \in E(G_{i+1})$, then this gives a rainbow Hamiltonian cycle of $\mathcal{G}$ (see Fig. \ref{lmm2.5switch}); (ii) If $xa_2 \notin E(G_{i+1})$, then  $ya_2\in E(G_{i+1})$ (as before) and we also replace $e_{i+1}$ by $ya_2$ and $e_{j}$ by $b_1x$ (see Fig. \ref{lmm2.5switch}) to obtain also in this final case a rainbow Hamiltonian cycle of $\mathcal{G}$.
\end{proof}

\begin{figure}[ht]
	\centering
	\includegraphics[scale=0.45] {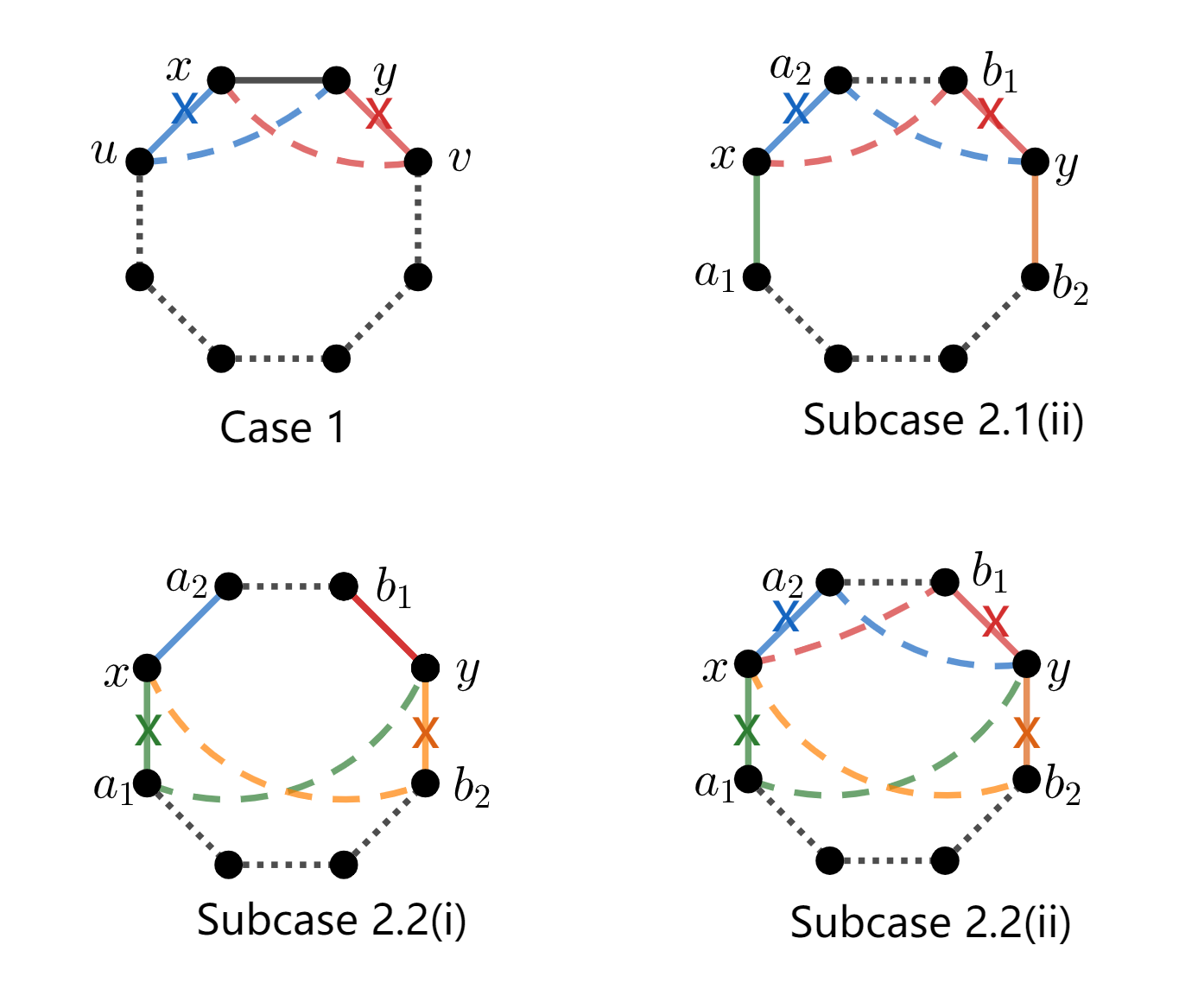}
	\caption{Switches in Lemma \ref{shift_ham}}
	\label{lmm2.5switch}
\end{figure}

We observe that the reverse of Lemma \ref{shift_ham} does not necessarily hold. For example, the family of graphs $G_1=\cdots=G_4\cong C_4$ has a rainbow Hamiltonian cycle, but $\mathcal{K}(G_1)=\cdots=\mathcal{K}(G_4)\cong K_1\vee (K_2\cup K_1)$ does not admit a rainbow Hamiltonian cycle.

{In order to prove Theorem \ref{thmsize} and Theorem \ref{thmrho}, we also use the following lemma. For $n\ge 4$, we define $e_j=\{j,n+2-j\}$ and $e_k'=\{k,n+1-k\}$ for $2\le j\le \lceil \tfrac{n}{2} \rceil$ and $1\le k\le \lfloor \tfrac{n}{2} \rfloor$.}

\begin{lem}\label{cltolem}
    Let {$n\ge 4$} and $G$ be a graph on vertex set $[n]$, with $e(G)> \binom{n-1}{2}+1$ or $\rho(G)>n-2$. Then 
    \begin{enumerate}[(i)]
        \item $\{e_3,\ldots, e_{\lceil \tfrac{n}{2} \rceil} \}\subseteq E(\mathcal{K}(G))$; \label{cltolem1}
        \item $\{e_1',\ldots, e_{\lfloor \tfrac{n}{2} \rfloor}' \}\subseteq E(\mathcal{K}(G))$. \label{cltolem2}
    \end{enumerate}
\end{lem}
\begin{proof}
Recall that if $\rho(G)>n-2$, then $e(G)\ge \binom{n-1}{2}+1$ by Stanley's inequality. Thus, in any case, $e(\mathcal{K}(G))=e(G)\ge  \binom{n-1}{2}+1$.

(i) Suppose that this claim does not hold, say $e_j \notin E(\mathcal{K}(G))$ for some $j$ with $3\le j \le \lceil \tfrac{n}{2} \rceil$. Then it follows from Lemma \ref{lem:kelordering} that if $z \ge j$ and $w \ge n+2-j$, then $\{z, w\} \notin  E(\mathcal{K}(G))$. If $(n,j) \neq (5,3)$, then it follows that $e(\mathcal{K}(G)) \le \binom{n}{2}-(n+2-2j)(j-1)-\binom{j-1}{2}= \binom{n-1}{2}+1 -(n-\tfrac{3}{2}j-\tfrac{1}{2})(j-2)<\binom{n-1}{2}+1$, a contradiction. For the case that $n=5$ and $j=3$, the above shows that $\mathcal{K}(G)$ is a subgraph of $K_{2}\vee 3K_1$, which has $7$ edges and spectral radius $3$, thus contracting the assumption that $e(G)> \binom{n-1}{2}+1$ or $\rho(G)>n-2$. Thus, \ref{cltolem1} holds.

(ii) It follows from \ref{cltolem1}  and Lemma \ref{lem:kelordering} that $e_k' \in E(\mathcal{K}(G))$ for $k \ge 3$. 
If $ e_1'=\{1,n\} \notin E(\mathcal{K}(G))$, then $n$ is an isolated vertex, and hence $\mathcal{K}(G)$ has size at most $\binom{n-1}{2}$, a contradiction.
	If $ e_2'=\{2,n-1\} \notin E(\mathcal{K}(G))$, then it follows that 
	$\mathcal{K}(G_i)$ is a subgraph of $K_1\vee (K_{n-3}\cup 2K_1)$, which has size 
    smaller than $\binom{n-1}{2}+1$,  a contradiction. Thus, also \ref{cltolem2} holds. 
\end{proof}

\section{Size conditions for rainbow Hamiltonicity} \label{Sec:size}

We are now ready to {prove} our first main result, which is a sufficient condition in terms of the size for rainbow Hamiltonicity of a graph family; a rainbow version of Ore's size condition (Theorem \ref{oresize}).
\begin{proof}[\textbf{Proof of Theorem \ref{thmsize}}]
	Assume (for a proof by contradiction) that $\mathcal{G}$ does not admit a rainbow Hamiltonian cycle.
	Then by Lemma \ref{shift_ham}, the family  $\mathcal{K}(\mathcal{G})$ does not admit a rainbow Hamiltonian cycle either, and $$e(\mathcal{K}(G_i))=e(G_i)> \binom{n-1}{2}+1 \ \mbox{for $i=1,2,\ldots,n.$}$$ 
Let $e_j=\{j,n+2-j\}$   for $2\le j\le \lceil \tfrac{n}{2} \rceil$ and $e_k'=\{k,n+1-k\}$ for $1\le k\le \lfloor \tfrac{n}{2} \rfloor$. 
{Then $\{e_3,\ldots, e_{\lceil \tfrac{n}{2} \rceil} \}\subseteq E(\mathcal{K}(G_i))$ and  $\{e_1',\ldots, e_{\lfloor \tfrac{n}{2} \rfloor}' \}\subseteq E(\mathcal{K}(G_i))$ for $i=1,2,\ldots,n$ by Lemma \ref{cltolem}.}

We now claim that $e_2 \notin E(\mathcal{K}(G_i))$ for $i=1,2,\ldots,n.$
In order to prove this claim by contradiction, we assume that $e_2 \in E(\mathcal{K}(G_i))$. Now it is relatively easy to construct a rainbow Hamiltonian cycle using this edge (with color $i$) and the edges $e_j$ $(3 \le j \le  \lceil \tfrac{n}{2} \rceil)$ and $e_k'$ $(1 \le k \le  \lfloor \tfrac{n}{2} \rfloor)$ (all of which we can color as we wish). In addition, we use an edge $\{j,1\}$ (which is an edge in all $\mathcal{K}(G_i)$ because $e_1'$ is so). For $n$ even, the cycle (indicated by its edges for clarity) is
	$$e_2\rightarrow e_2'\rightarrow e_3\rightarrow e_3'\rightarrow \cdots \rightarrow e_{\lceil \tfrac{n}{2} \rceil}\rightarrow e_{\lfloor \tfrac{n}{2} \rfloor}'\rightarrow \{n+1-\lfloor \tfrac{n}{2} \rfloor,1\}\rightarrow e_1'(\rightarrow e_2).$$
	For $n$ odd, the cycle is 
	$$e_2\rightarrow e_2'\rightarrow e_3\rightarrow e_3'\rightarrow \cdots \rightarrow e_{\lceil \tfrac{n}{2} \rceil}\rightarrow \{\lceil \tfrac{n}{2} \rceil,1\}\rightarrow e_1'(\rightarrow e_2).$$
	Thus, we have a contradiction, and the claim is proven.
	
	Now it follows from this last claim that vertex $n$ can only be adjacent to vertex $1$ in $\mathcal{K}(G_i)$ and hence $\mathcal{K}(G_i)$ is a subgraph of $K_{1}\vee (K_{n-2}\cup K_1)$. Thus, $\mathcal{K}(G_i)$ has at most $\binom{n-1}{2}+1$ edges, which is our final contradiction.
\end{proof}

We obtained a size condition of the rainbow Hamiltonicity of a graph family in Theorem \ref{thmsize} but did not characterize the extremal graphs. The reason is that the Kelmans transformation does not change the sizes of graphs, and we are not able to trace back the original graphs in $\mathcal{G}$ from $\mathcal{K}(G)$ by our approach. 

As mentioned in the introduction, if we allow $e(G_i) \ge \binom{n-1}{2}+1$, then for each $n \ge 6$, there is one exceptional graph, $K_1\vee (K_{n-2}\cup K_1)$, that is not Hamiltonian, and hence the graph family $\mathcal{G}=\{G_1,G_2,\ldots,G_n\}$ with $G_1=\cdots=G_n\cong K_1\vee (K_{n-2}\cup K_1)$ (see Fig.~\ref{extgraphs}) does not admit a rainbow Hamiltonian cycle. However, any other graph family with graphs isomorphic to $K_1\vee (K_{n-2}\cup K_1)$ is rainbow Hamiltonian, as we shall show next.

\begin{prop}\label{extgraphlem}
	Let {$n \ge 4$} and $\mathcal{G}=\{G_1,G_2,\ldots,G_n\}$ be a family of graphs of order $n$ on the same vertex set, where $G_i\cong K_1\vee (K_{n-2}\cup K_1)$ for each $1\le i \le n$. Then $\mathcal{G}$ admits a rainbow Hamiltonian cycle unless $G_1=G_2=\cdots =G_n$.
\end{prop}
\begin{proof}
	As remarked, if $G_1 = G_2=\cdots= G_n $ then $\mathcal{G}$ is not rainbow Hamiltonian since $ K_1\vee (K_{n-2}\cup K_1)$ is not Hamiltonian.  We assume now that the graphs in $\mathcal{G}$ are not equal and prove that $\mathcal{G}$ admits a rainbow Hamiltonian cycle.
	Let us denote the vertex degree of vertex $v$ in $G_i$ by $d_i(v)$.
	Suppose that  $G_1,\ldots,G_n$ have   $k$ distinct vertices with degree $1$, and denote these $k$ distinct vertices by $v_1,\ldots,v_k$. We distinguish the cases $k= 1$ and $k\ge 2$.

	We first give the proof for the case $k=1$. Then  all graphs in  $\mathcal{G}$  have $v_1$ with degree $1$. Without loss of generality, we assume that $G_1$ and $G_2$ are different. In this case, this means that the vertex that is adjacent to $v_1$ in $G_1$, $u_1$ say, and the one that is adjacent in $G_2$, $u_2$ say, are different. 
	Since the graph induced on $V\backslash \{v_1\}$ in $G_i$ is the same $K_{n-1}$ for each $1\le i \le n$, it is clear that by using edges $v_1u_1$ in $G_1$ and $v_1u_2$ in $G_2$, we can make a rainbow Hamiltonian cycle in $\mathcal{G}$.

	Next, we consider the case $k\ge 2$. Recall that $v_1,\ldots,v_k$ are the $k$ distinct vertices with degree $1$ in $G_1, \ldots, G_n$. 
	Let  $$\mathcal{M}_j=\{G_i: d_{i}(v_j)=1\}$$ and $m_j=|\mathcal{M}_j|$ for $1\le j \le k$.  
	We let $m_1\le m_2\le  \cdots \le m_k $ by reordering vertices $v_1,\ldots,v_k$ if necessary. Further,
	without loss of generality, we may assume that
	$\mathcal{M}_1=\{G_1,\ldots ,G_{m_1}\},$ 
	$\mathcal{M}_2=\{G_{m_1+1},\ldots ,G_{m_1+m_2}\},\ldots,$
	$\mathcal{M}_{k}=\{G_{m-m_k+1},\ldots ,G_{n}\}.$
	
	Now we let $e_i=v_{i+1}v_{i+2}$, then it follows that $e_i\in E(G_{i})$ for each $1\le i \le n-2$, and hence these edges form a rainbow path $P_{n-2}^R$ (from $v_2$ to $v_n$) in $\{G_1,\ldots,G_{n-2}\}$. 
	We aim to construct a rainbow  Hamiltonian cycle from this path by adding $v_1$ (on two edges, in $G_{n-1}$ and $G_n$) and replacing some edges.
	We will use that $d_{i}(v_n)\neq 1$ for $1\le i\le n-1$ and $d_{n}(v_1)\neq 1$ since $k\ge 2$.
	
	We first claim that  $v_1v_n\in E(G_{n-1})$, otherwise, we have $d_{n-1}(v_1)=1$ because $d_{n-1}(v_n)\neq 1$, which implies that $m_1=n-1$ and $m_2= 1$, a contraction to $m_1\le m_2$.
	
	If $v_1v_2\in E(G_n)$, then adding $e_{n-1}=v_1v_n$ and $e_n=v_1v_2$ to $P_{n-2}^R$ results in a  rainbow Hamiltonian cycle in $\mathcal{G}$. Otherwise, we get $v_1v_2\notin E(G_n )$, which implies that  $d_{n}(v_2)=1$, since $k\ge 2$.  Therefore, $k=2$ and $m_1+m_2=n$. Because $m_1 \le m_2$, this implies that $m_1<n-1$.
	Note that $d_{m_1}(v_1)=1$. Then $v_2v_{m_1+2}\in E(G_{m_1})$. 
	Thus, except for the case $m_1=1$, we have $v_1v_{m_1+1}\in E(G_n)$ since $d_{n}(v_2)=1$. Adding  $e_{n-1}=v_1v_n$ and $e_n=v_1v_{m_1+1}$, replacing $e_{m_1}=v_{m_1+1} v_{m_1+2}$ by $e_{m_1}'=v_2v_{m_1+2}$ and leaving all other edges in $P_{n-2}^R$ results in a  rainbow Hamiltonian cycle in $\mathcal{G}$  (see Fig.~\ref{lmm3.1switch} (a)).  For the remaining case ($m_1=1$ and $v_1v_2 \notin E(G_n)$), suppose that $v_2v_j\in E(G_{n})$ where $j\neq 1,2$. Adding $e_{n-1}=v_1v_n$ and  $e_{n}=v_2v_{j}$, and replacing $e_{j-2}=v_{j-1}v_{j}$ by $e_{j-2}'=v_1v_{j-1}$ and leaving all other edges in $P_{n-2}^R$ results in a  rainbow Hamiltonian cycle in $\mathcal{G}$  (see Fig.~\ref{lmm3.1switch} (b)).
\end{proof}

\begin{figure}[H]\centering\includegraphics[scale=0.43] {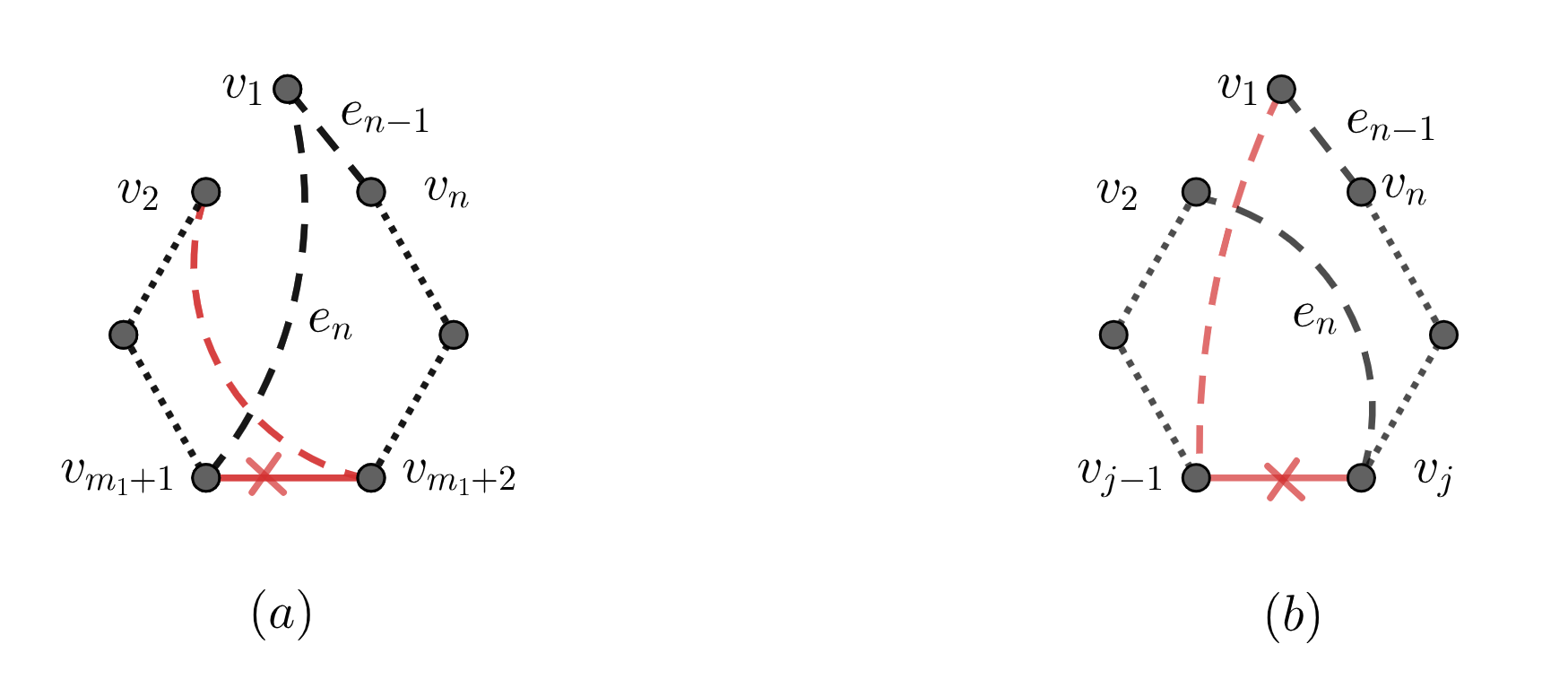}\caption{Rainbow Hamiltonian cycles in Proposition \ref{extgraphlem}}\label{lmm3.1switch}\end{figure}

\section{Spectral radius conditions for rainbow Hamiltonicity}\label{sec:spectralradius}

In order to prove Theorem \ref{thmrho}, we need the following lemma.

\begin{lem}\label{lem:specradiussubG} Let $n \ge 4$ and $G^*\cong K_1\vee (K_{n-2}\cup K_1)$.
	Then $\rho(G) \le n-2$ for every proper subgraph $G$ of $G^*$, with equality only if $G \cong K_{n-1}$ or $G \cong K_{n-1}\cup K_1$.
\end{lem}
\begin{proof}
	To show the bound, it suffices to show that it hold for graphs $G$ that are obtained from $G^*$ by deleting one edge. If this edge is incident with the vertex in $G^*$ that has degree $n-1$, then the maximum degree in $G$ is $n-2$, and hence it follows that $\rho(G) \le n-2$, with equality only if $G$ has a regular component of degree $n-2$; in our case, if $G \cong K_{n-1}\cup K_1$. Further, it is clear that the only other subgraph of this graph with spectral radius $n-2$ is $K_{n-1}$.
    
	In any other case, deleting an edge from $G^*$ results in a graph that admits an equitable partition into four parts with quotient matrix
	$$  Q=  \small{\begin{bmatrix}
			0&1&0&0\\
			1&0&2&n-4\\
			0&1&0&n-4\\
			0&1&2&n-5
	\end{bmatrix}},$$
	that is, for $n \ge 5$ (the case $n=4$ gives $K_{1,3}$, with spectral radius $\sqrt {3}$). It is straightforward to check that if $\psi_n$ is the characteristic polynomial of $Q$, then $\psi_n(n-2)>0$ and $\psi_n(n-3)<0$. Unless $Q$ has three eigenvalues that are larger than $n-3$, this shows that $\rho(G)<n-2$. But the trace of $Q$ equals $n-5$, so three eigenvalues larger than $n-3$ would imply that the remaining eigenvalue is less than $-n$, which is impossible.    
\end{proof}

Now we have all the tools to extend the proof of Theorem \ref{thmsize} and prove Theorem \ref{thmrho}.

\begin{proof}[\textbf{Proof of Theorem \ref{thmrho}}]
	We again write $G^*\cong K_1\vee (K_{n-2}\cup K_1)$, and note that
	if $G_1 = G_2=\cdots= G_n \cong  G^*$, then $\mathcal{G}$ is not rainbow Hamiltonian. 
	
	Next, suppose that $\mathcal{G}$ is not rainbow Hamiltonian with $\rho(G_i)>n-2$ for each $i$. By Lemma \ref{shift_ham}, $\mathcal{K}(\mathcal{G})$ does not admit a rainbow Hamiltonian cycle either and by Lemma \ref{ktrho}, we have $\rho (\mathcal{K}(G_i))\ge \rho (G_i)> n-2$ for each $i$. 
    By the same arguments as in the proof of Theorem \ref{thmsize}, it now follows that  
     $\mathcal{K}(G_i)$ is a subgraph of $G^*$. Because $\rho (\mathcal{K}(G_i))> n-2$, it then follows from Lemma \ref{lem:specradiussubG} that  $\mathcal{K}(G_i)\cong G^*$, for each $i$. Finally, by Proposition \ref{extgraphlem}, we then have that
	$G_1=\cdots=G_n\cong G^*.$  \end{proof}


\section{The signless Laplacian spectral radius}\label{sec:signless}

The \textit{signless Laplacian matrix} of $G$ is defined as $S(G) = D(G) + A(G)$, where $D(G)$ is the diagonal matrix of vertex degrees of $G$. The largest eigenvalue of  $S(G)$, denoted by $\rho_S(G)$, is called the \textit{signless Laplacian spectral radius} of $G$. Note that $S(G)=2A(G) + L(G)$, where $L(G)$ is the usual Laplacian matrix, which is positive semidefinite. Hence $\rho_S(G) \ge 2\rho(G)$.

Yu and Fan \cite{yu2013spectral} observed that Theorem \ref{fnspectral} has a (slightly stronger) analogue in terms of the signless Laplacian spectral radius. They showed that $\rho_S(G) > 2n-4$ implies that $e\ge \binom{n-1}{2}+1$, and hence they obtained the following from Theorem \ref{bondysize}.

\begin{theorem}\cite[Theorem 4.1]{yu2013spectral}\label{yufan}
	Let $n \ge 6$ and $G$ be a graph of order $n$. If
	$\rho_S(G)>2n-4,$
	then $G$ is Hamiltonian unless $G\cong K_1\vee (K_{n-2}\cup K_1).$
\end{theorem}

We note that $n=5$ is again excluded, because the non-Hamiltonian graph  $K_2\vee 3K_1$ has signless spectral radius larger than 6.

A rainbow version of this result is now easily obtained by adjusting our proof of Theorem \ref{thmrho} (and the underlying lemmas) to the signless spectral radius.

\begin{theorem}\label{thmsignless}
	Let $n\ge 6$ and $\mathcal{G}=\{G_1,G_2,\ldots,G_n\}$ be a family of graphs on the same vertex set $[n]$ with
	\begin{align*}
		\rho_S(G_i) > 2n-4,  \mbox{\ for\ $i=1,2,\ldots,n.$}
	\end{align*}
	Then $\mathcal{G}$ admits a rainbow Hamiltonian cycle unless $G_1 = G_2=\cdots= G_n \cong  K_1\vee (K_{n-2}\cup K_1).$
\end{theorem}

\section*{Acknowledgements}
The authors would like to thank the anonymous referees for their helpful comments.

\bibliography{rbh}
\bibliographystyle{plain}	

\end{spacing}
\end{document}